\documentclass[10pt]{article}
\usepackage{amsmath, amsfonts, amsthm, amssymb,verbatim}
        \newtheorem{theorem}{Theorem}[section]
\newtheorem{definition}[theorem]{Definition}
        \newtheorem{lemma}[theorem]{Lemma}
        \newtheorem{corollary}[theorem]{Corollary}

\numberwithin{equation}{section}
\newcommand \Ucal {\mathcal U}
\newcommand \Qbf {\mathbf Q}
\newcommand \ub {\overline u}
\newcommand \vb {\overline v}
\newcommand \omegab {{\widetilde \omega}}
\newcommand \ubar   	{{\overline u}}
\newcommand \barc {\underline c}
\newcommand \cbar {\overline c}
\newcommand \dive 	{\mbox{div}}

\newcommand \del		\partial
\newcommand \Ecal {\mathcal E}

\newcommand{\auth}{\textsc}
\newcommand \R     {\mathbb{R}}

\newcommand \MM     {M}

\newcommand \vbar   {{\overline v}}

\DeclareMathOperator	\sgn  {sgn}

\newcommand \supp {\text{supp }}

\newcommand \TT		{\mathcal{T}}
\newcommand \Tcal	{\TT}
\newcommand{\dK}{\partial K}

\newcommand{\sumk} {\sum_{K\in\TT^h}}

\newcommand {\sumkez} {\sum_{\substack{ K \in\Tcal^h\\ \ekz\in\dKz}}}

\newcommand{\Hkez}{H_{K,e^{0}}}

\newcommand{\qq}{{{q}_{K,\ekz}}}

\newcommand{\ukp}{u_{K}^{+}}

\newcommand{\ukm}{u_{K}^-}
\newcommand{\ukezm}{u_{K_{e^0}}^-}
\newcommand{\tu}{\widetilde{u}_{K,e^0}^+}

\newcommand \la 		\langle
\newcommand \ra 		\rangle
\newcommand \be   	{\begin{equation}}
\newcommand \ee   	{\end{equation}}
\newcommand \RR      {\mathbb{R}}
\newcommand \eps     \epsilon

\newcommand \om       {\omega}
\newcommand \Om      {\Omega}
\newcommand \al \alpha 
\newcommand \var      {\varphi}

\newcommand{\ekp}{e_{K}^{+}}
\newcommand{\ekm}{e_{K}^{-}}

\newcommand{\ekz}{e^{0}}

\newcommand{\dKz}{\del^{0}K}
\newcommand{\vep}{ \varphi_{e_K^+}}
\newcommand{\vem}{ \varphi_{e_K^-}}
\newcommand{\vepo}{ \varphi_{e_K^+}^{\omega}}

\newcommand{\vepO}{ \varphi_{e_K^+}^{\Omega}}
\newcommand{\vemO}{ \varphi_{e_K^-}^{\Omega}}
\newcommand{\veo}{ \varphi_{e}^{\omega}}
\newcommand{\veO}{ \varphi_{e}^{\Omega}}

\newcommand{\varbO} {\var_{\ekp}^\Omega}
\newcommand{\QQ}{{Q}_{K,\ekz}}
\newcommand \Omegabf {\mbox{\boldmath $\Omega$}}

\newcommand{\sumK}{\sum_{e^{0}\in\del^{0}K}}

\newcommand \iht { i^*_{H_t}}

\newcommand \ihz {i^*_{H_0}}

\def\Xint#1{\mathchoice
{\XXint\displaystyle\textstyle{#1}}%
{\XXint\textstyle\scriptstyle{#1}}%
{\XXint\scriptstyle\scriptscriptstyle{#1}}%
{\XXint\scriptscriptstyle\scriptscriptstyle{#1}}%
\!\int}
\def\XXint#1#2#3{{\setbox0=\hbox{$#1{#2#3}{\int}$}
\vcenter{\hbox{$#2#3$}}\kern-.5\wd0}}

\def\dashint{\Xint\diagup}

\begin{document}

\setcounter{footnote}{-1}

\title{Hyperbolic conservation laws on spacetimes.
\\
A finite volume scheme based on differential forms}

\author{
Philippe G. LeFloch and Baver Okutmustur
\footnote{Laboratoire Jacques-Louis Lions \& Centre National de la Recherche Scientifique,
Universit\'e Pierre et Marie Curie (Paris 6), 4 Place Jussieu,  75252 Paris, France.
\newline
E-mail : {\sl pgLeFloch@gmail.com, Okutmustur@ann.jussieu.fr}
To appear in: Far East Journal of Mathematical Sciences (2008). 
}}

\date{September 2008}
\maketitle
\begin{abstract}
We consider nonlinear hyperbolic conservation laws, 
posed on a differential $(n+1)$-manifold with boundary referred to as a {\sl spacetime,}
and in which the ``flux'' is defined as a {\sl flux field of $n$-forms} depending on a parameter
(the unknown variable).
We introduce a formulation of the initial and boundary value problem which
 is geometric in nature and is more natural than the vector field approach
recently developed for Riemannian manifolds.
Our main assumption on the manifold and the flux field is a {\sl global hyperbolicity} condition,
which provides a global time-orientation as is standard in Lorentzian geometry and general relativity.
Assuming that the manifold admits a foliation by
compact slices, we establish the existence of a semi-group of entropy solutions.
Moreover, given any two hypersurfaces with one lying in the future of the other,
we establish a ``contraction'' property which compares two entropy solutions,
in a (geometrically natural) distance equivalent to the $L^1$ distance.
To carry out the proofs, we rely on a new version of the {\sl finite volume method,}
which only requires the knowledge of the given $n$-volume form structure on the $(n+1)$-manifold
and involves the {\sl total flux} across faces of the elements of the triangulations, only,  
rather than the product of a numerical flux times the measure of that face.
\end{abstract}


\section{Introduction}

The development of the mathematical theory (existence, uniqueness, qualitative behavior, approximation)
of shock wave solutions
to scalar conservation laws {\sl defined on manifolds}
is motivated by similar questions arising in compressible fluid dynamics.
For instance, the shallow water equations of geophysical fluid dynamics
(for which the background manifold is the Earth or, more generally, Riemannian manifold) 
and the Einstein-Euler equations in general relativity
(for which the manifold metric is also part of the unknowns)
provide important examples where the partial differential equations
 of interest are naturally posed on a (curved) manifold.
Scalar conservation laws yield a drastically simplified, yet very challenging, mathematical
model for understanding nonlinear aspects of shock wave propagation on manifolds.

In the present paper, given a (smooth) differential $(n+1)$-manifold $M$
 which we refer to as a {\sl spacetime,}
we consider the following class of {\sl nonlinear conservation laws}
\be
\label{cons1}
d(\om(u)) = 0, \qquad u=u(x), \, x \in M.
\ee
Here, for each $\ubar \in \RR$, $\omega=\omega(\ubar)$ is a (smooth) {\sl field of $n$-forms} 
on $M$ which we refer to as the {\sl flux field} of the conservation law \eqref{cons1}.

Two special cases of \eqref{cons1} were recently studied in the literature.
When $M = \RR_+ \times N$ and the $n$-manifold $N$ is 
endowed with a {\sl Riemannian metric} $h$, the conservation law \eqref{cons1} is here equivalent to
$$
\del_t u + \dive_h (b(u)) = 0, \qquad u=u(t,y), \ t \geq 0, \, y \in N.
$$
Here, $\dive_h$ denotes the divergence operator associated with the metric $h$.
In this case, the flux field is a {\sl flux vector field} $b=b(\ubar)$ {\sl on the $n$-manifold $N$}
and does not depend on the time variable. 
More generally, we may suppose that $M$ is endowed with a {\sl Lorentzian metric} $g$
and, then, \eqref{cons1} takes the equivalent form
$$
\dive_g (a(u)) = 0, \qquad u=u(x), \, x \in M.
$$
Observe that the flux $a=a(\ubar)$ is now a vector field {\sl on the $(n+1)$-manifold $M$.}

Recall that, in the Riemannian or Lorentzian settings, the theory of weak solutions on manifolds
was initiated by Ben-Artzi and LeFloch \cite{BL} and further developed in the follow-up papers by LeFloch and his collaborators 
\cite{ABL,ALO,PLF,LNO,LO}.
Hyperbolic equations on manifolds were also studied by Panov in \cite{Panov} with a vector field standpoint. 
The actual implementation of a finite volume scheme on the sphere was recently realized by 
Ben-Artzi, Falcovitz, and LeFloch \cite{BFL1}. 

In the present paper, we propose a new approach in which the conservation law is written in the form \eqref{cons1}, that is, 
the flux $\omega=\omega(\ubar)$ is defined as a {\sl field of differential forms of degree $n$.}
Hence, no geometric structure is a~priori assumed on $M$, and the sole knowledge of the flux field structure
is required. The fact that the equation \eqref{cons1} is a ``conservation law'' for the unknown quantity $u$
can be understood by expressing Stokes theorem:
for sufficiently smooth solutions $u$, at least, the conservation law \eqref{cons1} is equivalent to
saying that the total flux
\be
\label{486}
\int_{\del \Ucal} \om(u) = 0, \qquad \Ucal \subset M,
\ee
vanishes for every open subset $\Ucal$ with smooth boundary.
By relying on the conservation law \eqref{cons1} rather than the equivalent expressions in the special cases 
of Riemannian or Lorentzian manifolds,
we are able to develop here a theory of entropy solutions to conservation laws posed on manifolds,
which is technically and conceptually simpler 
but also provides a significant generalization of earlier works.

Recall that weak solutions to conservation laws contain shock waves and, for the sake of uniqueness, 
the class of such solutions must be restricted by an entropy condition (Lax \cite{Lax}).
This theory of conservation laws on manifolds is a generalization of fundamental works
by Kruzkov \cite{Kruzkov}, Kuznetsov \cite{Kuznetsov}, and DiPerna \cite{DiPerna} who treated
equations posed on the (flat) Euclidian space $\R^n$.

Our main result in the present paper is a generalization of the formulation and convergence of the 
finite volume method for general conservation law \eqref{cons1}. In turn, we will
establish the existence of a semi-group of entropy solutions which is contracting in a suitable distance.

The first difficulty is formulating the initial and boundary problem for \eqref{cons1} in the sense of distributions. 
A weak formulation of the boundary condition is proposed which takes into account the nonlinearity and hyperbolicity 
of the equation under consideration.
We emphasize that our weak formulation applies to an arbitrary differential manifold.
However, to proceed with the development of the well-posedness theory we then need to impose
that the manifold satisfies a {\sl global hyperbolicity condition,} which provides a global time-orientation
and allow us to distinguish between ``future'' and ``past'' directions in the time-evolution.
This assumption is standard in Lorentzian geometry for applications to general relativity. For simplicity in this paper, we then
restrict attention to the case that the manifold is foliated by compact slices.

Second, we introduce a new version of the finite volume method (based on monotone numerical flux terms).
The proposed scheme provides a natural discretization of the conservation law
\eqref{cons1}, which solely uses the $n$-volume form structure
associated with the prescribed flux field $\omega$.

Third, we derive several stability estimates satisfied by the proposed scheme, especially
discrete versions of the entropy inequalities. As a corollary, we obtain a uniform control of the
entropy dissipation measure associated with the scheme, which, however, is not sufficient by itself to
the compactness of the sequence of approximate solutions.

The above stability estimates are sufficient to show that the sequence of approximate solutions generated
by the finite volume scheme converges to an {\sl entropy measure-valued solution} in the sense of DiPerna.
To conclude our proof, 
we rely on a generalization of DiPerna's uniqueness theorem \cite{DiPerna} and 
conclude with the
existence of entropy solutions to the corresponding initial value problem.

In the course of this analysis, we also establish a {\sl contraction property}
for any two entropy solutions $u, v$, that is,
given two hypersurfaces $H, H'$ such that $H'$ lies in the future of $H$,
\be
\label{489}
\int_{H'} \Omegabf(u_{H'}, v_{H'}) \leq \int_{H} \Omegabf(u_{H}, v_{H}).
\ee  
Here, for all reals $\ubar, \vbar$, the $n$-form field
$\Omegabf(\ubar, \vbar)$ is determined from the given flux field $\omega(\ubar)$
and can be seen as a generalization (to the spacetime setting) of the notion of Kruzkov entropy $|\ubar - \vbar|$.

Recall that DiPerna's measure-valued solutions were used to establish the convergence of schemes by Szepessy \cite{Szepessy,Szepessy2},
 Coquel and LeFloch \cite{CL1,CL2,CL3}, and Cockburn, Coquel, and LeFloch \cite{CCL00,CCL2}. 
For many related results and a review about the convergence techniques for hyperbolic problems, 
we refer to Tadmor \cite{Tadmor} and Tadmor, Rascle, and Bagneiri \cite{TRB}. 
Further hyperbolic models including also a coupling with elliptic equations and many 
applications were successfully 
investigated in the works by Kr\"oner \cite{Kroener}, and Eymard, 
Gallouet, and Herbin \cite{EGH}. For higher-order schemes, see the paper by Kr\"oner, Noelle, and Rokyta \cite{KNR}. 
Also, an alternative approach to the convergence of finite volume schemes
was later proposed by Westdickenberg and Noelle \cite{WN}. Finally, 
note that Kuznetsov's error estimate \cite{CCL00,CCL1} 
were recently extended to conservation laws on manifolds by LeFloch, Neves, and Okutmustur \cite{LNO}.

An outline of the paper is as follows. In Section~\ref{deux},
we introduce our definition of entropy solution which includes both initial-boundary data and entropy inequalities.
The finite volume method is presented in Section~\ref{finit}, and discrete stability properties
are then established in Section~\ref{discret}.
The main statements 	are given at the beginning of Section~\ref{wellp}, together with 
the final step of the convergence proof.


\section{Conservation laws posed on a spacetime}
\label{deux}

\subsection{A notion of weak solution}

In this section we assume that $M$ is an oriented, compact, differentiable $(n+1)$-manifold
with boundary. 
Given an $(n+1)$-form $\alpha$, its {\sl modulus} is defined as the $(n+1)$-form
$$
|\alpha| : = |\overline{\al}|\, dx^0 \wedge \cdots \wedge dx^n,
$$
where $\alpha = \overline{\al} \,dx^1 \wedge \cdots \wedge dx^n$ is written in an oriented frame
determined from local coordinates $x=(x^\alpha)=(x^0, \ldots, x^n)$.
If $H$ is a hypersurface, we denote by $i=i_H : H \to \MM$ the canonical injection map, and
by $i^*=i_H^*$ is the pull-back operator acting on differential forms defined on $M$.

On this manifold, we introduce a class of nonlinear hyperbolic equations, as follows.

\begin{definition}

1. A {\rm flux field} $\omega$ on the $(n+1)$-manifold $M$
is a parametrized family $\omega(\ubar) \in \Lambda^n(M)$ of smooth fields
of differential forms of degree $n$,
that depends smoothly upon the real parameter $\ubar$.

2. The {\rm conservation law} associated with a flux field $\om$ and with unknown $ u:  M \to \RR$ is
\be
\label{LR.1}
d\big(\om(u)\big)=0,
\ee
where $d$ denotes the exterior derivative operator and, therefore, $d\big(\om(u)\big)$ is a field
of differential forms of degree $(n+1)$ on $M$.

3. A flux field $\omega$ is said to {\rm grow at most linearly}
if for every $1$-form $\rho$ on $M$
\be
\label{LR.2}
\sup_{\ubar \in \RR} \int_M \left| \rho \wedge \del_u \om(\ubar) \right| < \infty.
\ee
\end{definition}

With the above notation, by introducing local coordinates $x=(x^\al)$ we can write for all $\ubar \in \RR$
$$
\aligned
\om(\ubar) & = \om^\al(\ubar) \,  (\widehat{dx})_\al,
\\
(\widehat{dx})_\al  & := dx^0 \wedge \ldots \wedge dx^{\al-1} \wedge dx^{\al+1} \wedge \ldots \wedge dx^n.
\endaligned
$$
Here, the coefficients $\omega^\alpha = \om^\al(\ub)$ are smooth functions defined in the 
chosen local chart.
Recall that the operator $d$ acts on differential forms with arbitrary
degree and that, given a $p$-form $\rho$ and a $p'$-form $\rho'$, one has $d(d\rho)=0$ and
$d(\rho \wedge \rho')= d\rho \wedge \rho' +(-1)^p \rho \wedge d\rho'$.

As it stands, the equation \eqref{LR.1} makes sense for unknown functions that are, for instance,
Lipschitz continuous. However, it is well-known that solutions to nonlinear hyperbolic equations
need not be continuous and, consequently, we need to recast \eqref{LR.1} in a weak form.

Given a {\sl smooth} solution $u$ of \eqref{LR.1} we can apply Stokes theorem on any open subset $\Ucal$
that is compactly included in $M$ and has smooth boundary $\del \Ucal$. We obtain
\be
\label{key64}
0 = \int_\Ucal d(\om(u)) = \int_{\del \Ucal} i^*(\om(u)).
\ee
Similarly, given any smooth function $\psi: M \to \RR$ we can write
$$
d(\psi \, \om(u)) = d\psi \wedge \om(u) + \psi \, d(\om(u)),
$$
where the differential $d\psi$ is a $1$-form field. Provided $u$ satisfies \eqref{LR.1}, we find
$$
\int_M d(\psi \, \om(u)) = \int_M d\psi \wedge \om(u)
$$
and, by Stokes theorem,
\be
\label{LR.0}
\int_M d\psi\wedge \om(u)=\int_{\del \MM}i^*(\psi\om(u)).
\ee
Note that a suitable orientation of the boundary $\del M$ is required for this formula
to hold.
This identity is satisfied by every smooth solution to \eqref{LR.1} and this motivates us to reformulate
\eqref{LR.1} in the following weak form.

\begin{definition}[Weak solutions on a spacetime]
Given a flux field with at most linear growth $\omega$,
a function $u \in L^1(M)$ is called a {\rm weak solution} to the conservation law \eqref{LR.1}
posed on the spacetime $M$
if
$$
\int_M d\psi\wedge \om(u) = 0
$$
for every function $\psi : M \to \RR$ compactly supported in the interior $\mathring M$.
\end{definition}

The above definition makes sense since the function $u$ is integrable and $\om(\ub)$ has at most linear growth in $\ub$,
so that the $(n+1)$-form $d\psi\wedge \om(u)$ is integrable on the compact manifold $M$.  


\subsection{Entropy inequalities}

As is standard for nonlinear hyperbolic problems,
weak solution must be further constrained by imposing initial, boundary, as well as entropy conditions.

\begin{definition}
\label{key63}
A (smooth) field of $n$-forms $\Om=\Om(\ubar)$ 
is called a {\rm (convex) entropy flux field} for the conservation law \eqref{LR.1} if there exists
a (convex) function $U: \RR \to \RR$ such that
$$
\Om(\ubar) = \int_0^\ubar \del_u U (\vb) \, \del_u \om(\vb) \, d\vb, \qquad \ubar \in \RR.
$$
It is said to be also {\rm admissible} if, moreover, $\sup | \del_u U | < \infty$.
\end{definition}

For instance, if one chooses the function $U(\ub, \vb):=|\ub - \vb|$, where $\vb$
is a real parameter, the entropy flux field reads
\be
\label{KRZ}
\Omegabf(\ub, \vb) := \sgn(\ub - \vb) \, ( \om(\ub) - \om(\vb)),
\ee
which is a generalization to a spacetime of the so-called Kruzkov's entropy-entropy flux pair.

Based on the notion of entropy flux above, we can derive entropy inequalities in the following way.
Given any smooth solution $u$ to \eqref{LR.1}, by multiplying \eqref{LR.1} by $\del_u U(u)$
we obtain the additional conservation law
$$
d( \Om(u) ) -(d\Om)(u) + \del_u U(u) (d\om)(u) = 0.
$$
However, for discontinuous solutions this identity can not be satisfied as an equality and, instead,
we should impose that the entropy inequalities
\be
\label{LR.1i}
d( \Om(u)) - (d\Om)(u) + \del_u U (u) (d\om)(u) \leq 0
\ee
hold in the sense of distributions for all admissible entropy pair $(U,\Om)$.
These inequalities can be justified, for instance, via the vanishing viscosity method, that is
by searching for weak solutions that are realizable as limits of smooth solutions to the parabolic
regularization of \eqref{LR.1}.

It remains to prescribe initial and boundary conditions. We emphasize that, without further assumption
on the flux field (to be imposed shortly below), points along the boundary
$\del M$ can not be distinguished and it is natural to prescribe the trace of the solution
along the {\sl whole} of the boundary $\del M$. This is possible provided the boundary data, $u_B: \del M \to \RR$,
is assumed by the solution in a suitably {\sl weak sense}. Following Dubois and LeFloch \cite{DL}, 
we use the notation 
\be
u \big|_{\del M} \in \Ecal_{U,\Om}( u_B)
\label{LR.2i}
\ee
for all convex entropy pair $(U,\Om)$, where for all reals $\ub$ 
$$
\Ecal_{U,\Om} (\ub):= \Big\{ \vb \in \RR \, \, \big| \, \, E(\ub, \vb):= \Om(\ub) + \del_uU(\ub)(\om(\vb) - \om(\ub))
\leq  \Om(\vb)  \Big \}.
$$
Recall that the boundary conditions for hyperbolic conservation laws 
(posed on the Euclidian space) 
were first studied by Bardos, Leroux, and Nedelec \cite{BLN} in the class of solutions with bounded variation 
and, then, in the class of measured-valued solutions by Szepessy \cite{Sz}.
Later, a different approach was introduced by Cockburn, Coquel, and LeFloch \cite{CCL1}
(see, in particular, the discussion p.~701 therein) 
in the course of their analysis of the finite volume methods, 
which was later expanded in Kondo and LeFloch \cite{KL}.
An alternative and also powerful approach to the boundary conditions for conservation laws 
was independently introduced by Otto \cite{Otto} and developed by followers. 
In the present paper, our proposed
formulation of the initial and boundary value problem is a generalization of the works \cite{CCL1} and \cite{KL}.

\begin{definition}[Entropy solutions on a spacetime with boundary]
\label{key92}
Let $\om=\om(\ub)$ be a flux field with at most linear growth and
$u_B \in L^1(\del M)$ be a prescribed boundary function.
A function $u \in L^1(M)$ is called an {\rm entropy solution} to the boundary value problem $\eqref{LR.1}$
and $\eqref{LR.2i}$
if there exists a bounded and measurable field of $n$-forms $\gamma \in L^1\Lambda^n(\del M)$
such that
$$
\aligned
&\int_M \Big(  d \psi \wedge \Om(u)
+ \psi \, (d  \Om) (u) - \psi \, \del_u U(u) (d \om) (u) \Big)
\\
& + \int_{\del \MM} \psi_{|\del M} \, \big(  i^*\Om(u_B) + \del_u U(u_B)  \big(\gamma - i^*\om(u_B) \big) \big)
\,  \geq 0
\endaligned
$$
for every admissible convex entropy pair $(U,\Om)$ and every smooth function $\psi: M \to \RR_+$.
\end{definition}

Observe that the above definition makes sense since each of the terms
$d \psi \wedge \Om(u)$,  $(d  \Om) (u)$,  $(d \om) (u)$ belong to $L^1(M)$.
The above definition can be generalized to encompass solutions within the much larger class of
measure-valued mappings.
Following DiPerna \cite{DiPerna}, we consider solutions that are no longer functions but
{\sl Young measures,} i.e,
weakly measurable maps $\nu: M \to \text{Prob}(\RR)$ taking values
within is the set of probability measures $\text{Prob}(\RR)$.
For simplicity, we assume that the support $\supp \nu$ is a compact subset of $\RR$.

\begin{definition}
\label{LR.4}
Given a flux field $\om=\om(\ub)$ with at most linear growth and given
a boundary function $u_B \in L^\infty(\del M)$, one says that
a compactly supported Young measure $\nu: M \to \text{Prob}(\RR)$ is
an {\rm entropy measure-valued solution}
to the boundary value problem $(\ref{LR.1}), (\ref{LR.2i})$
if there exists a bounded and measurable field of $n$-forms $\gamma \in L^\infty\Lambda^n(\del M)$
such that the inequalities
$$
\aligned
& \int_M  \Big\la \nu,  d \psi \wedge \Om(\cdot)  +  \psi \, \big(d  (\Om (\cdot)) - \del_u U(\cdot) (d \om) (\cdot)\big) \Big\ra
\\
& + \int_{\del M} \psi_{|\del M} \, \Big\la \nu, \Big( i^*\Om(u_B) + \del_u U(u_B)  \big(\gamma - i^*\om(u_B)\big)\Big) \Big\ra
\,  \geq 0
\endaligned
$$
hold for all convex entropy pair $(U,\Om)$ and all smooth functions $\psi \geq 0$.
\end{definition}


\subsection{Global hyperbolicity and geometric compatibility}

In general relativity, it is a standard assumption that the spacetime should be globally hyperbolic.
This notion must be adapted to the present setting, since we do not have a Lorentzian structure, but solely
the $n$-volume form structure associated with the flux field $\omega$.

We assume here that the manifold $M$ is foliated by hypersurfaces, say
\be
\label{foli}
M = \bigcup_{0 \leq t \leq T} H_t,
\ee
where each slice has the topology of a (smooth)
$n$-manifold $N$ with boundary. Topologically we have $M = [0,T] \times N$, and the boundary
of $M$ can be decomposed as
\be
\label{foli2}
\aligned
\del M & = H_0 \cup H_T \cup B,
\\
B = (0,T) \times N & :=  \bigcup_{0 < t < T} \del H_t.
\endaligned
\ee

The following definition imposes a non-degeneracy condition on the averaged flux on the hypersurfaces of the foliation.

\begin{definition}
\label{hyperb-def}
Consider a manifold $M$ with a foliation \eqref{foli}-\eqref{foli2} and let
$\om=\om(\ubar)$ be a flux field. Then,
the conservation law \eqref{LR.1} on the manifold $M$ is said to satisfy the
{\rm global hyperbolicity condition} if there exist constants $0 < \barc < \cbar$ such that
for every non-empty hypersurface $e \subset H_t$,
the integral $\int_e i^*  \del_u \om(0)$ is positive and
 the function $\var_e: \RR \to \RR$,
$$
\var_e(\ubar) := \dashint_e i^* \om(\ubar)
               = {\int_e i^* \om(\ubar) \over \int_e i^*  \del_u \om(0)},
\qquad
\ubar \in \RR
$$
satisfies
\be
\label{hyperb}
\barc \leq \del_u \var_e(\ubar) \leq \cbar,  \qquad \ub \in \RR.
\ee
\end{definition}

The function $\var_e$ represents the {\sl averaged flux} along the hypersurface $e$.
From now we assume that the conditions in Definition~\ref{hyperb-def} are satisfied.
It is natural to refer to $H_0$ as an initial hypersurface
and to prescribe an ``initial data'' $u_0 : H_0 \to \RR$ on this hypersurface and, on the other hand,
to impose a boundary data $u_B$ along the submanifold $B$. It will
be convenient here to use the standard terminology
of general relativity and to refer to $H_t$ as {\sl spacelike hypersurfaces.}

Under the global hyperbolicity condition \eqref{foli}--\eqref{hyperb}, the initial and boundary value problem now takes
the following form. The boundary condition \eqref{LR.2i} decomposes into
an initial data
\be
\label{cond1}
u_{H_0} = u_0
\ee
and a boundary condition
\be
\label{cond2}
u \big|_B \in \Ecal_{U,\Om}( u_B).
\ee
Correspondingly, the condition in Definition~\ref{key92} now reads
$$
\aligned
&\int_M \Big(  d \psi \wedge \Om(u)
+ \psi \, (d  \Om) (u) - \psi \, \del_u U(u) (d \om) (u) \Big)
\\
& + \int_B \psi_{|\del M} \, \big(  i^*\Om(u_B) + \del_u U(u_B)  \big(\gamma - i^*\om(u_B) \big) \big)
\\
& + \int_{H_T} i^* \Omega(u_{H_T}) - \int_{H_0} i^* \Omega(u_0) \geq 0.
\endaligned
$$

Finally, we introduce: 

\begin{definition}
A flux field $\om$ is called {\rm geometry-compatible} if it is closed for each value of the parameter,
\be
\label{LR.3}
( d \om ) (\ubar)=0, \qquad \ubar \in \RR.
\ee
\end{definition}

This compatibility condition is natural since it ensures that constants are trivial solutions to the conservation law, 
a property shared by many models of fluid dynamics 
(such as the shallow water equations on a curved manifold). 
When \eqref{LR.3} holds, then it follows from Definition~\ref{key63} that
every entropy flux field $\Om$ also satisfies the condition
$$
(d\Om) (\ubar) =0, \qquad \ubar \in \RR.
$$
In turn, the entropy inequalities \eqref{LR.1i} for a solution $u : M \to \RR$ simplify drastically and take the form
\be
\label{LR.1i-simple}
d( \Om(u)) \leq 0.
\ee


\section{Finite volume method on a spacetime}
\label{finit}

\subsection{Assumptions and formulation}

From now on we assume that the manifold $M= [0,T] \times N$ is foliated by slices with {\sl compact} topology $N$,
and the initial data $u_0$ is taken to be a bounded function.
We also assume that the global hyperbolicity condition holds and 
that the flux field $\om$ is geometry-compatible, which simplifies the presentation
but is not an essential assumption.

Let $\TT^h = \bigcup_{K \in\TT^h} K$ be a {\sl triangulation} of the manifold $M$, that is,
a collection of finitely many {\sl cells (or elements),} determined as the images of polyhedra of $\R^{n+1}$,
satisfying the following conditions:
\begin{itemize}
\item The boundary $\dK$ of an element $K$ is a piecewise smooth, $n$-manifold,
$\dK = \bigcup_{e\subset \dK} e$ and  contains exactly two 
spacelike faces,
denoted by  $\ekp$ and $\ekm$, and ``vertical'' elements
$$
e^0 \in \del^0 K:= \del K \setminus \big\{\ekp, \ekm\big\}.
$$
\item The intersection $K \cap K'$ of two distinct elements $K, K' \in \TT^h$
is either a common face of $K, K'$ or else a submanifold with dimension at most $(n-1)$.
\item The triangulation is compatible with the foliation \eqref{foli}-\eqref{foli2} in the sense that
there exists a sequence of times $t_0= 0 < t_1 < \ldots < t_N = T$ such that all spacelike faces
are submanifolds of $H_n := H_{t_n}$ for some $n=0, \ldots, N$,
and determine a triangulation of the slices.
We denote by $\Tcal^h_0$ the set of all elements $K$ which
admit one face belonging to the initial hypersurface
$H_0$.

\end{itemize}

We define the measure $|e|$ of a hypersurface $e \subset M$ by
\be
\label{e-def}
|e|:= \int_e i^* \del_u \om(0).
\ee
This quantity is positive if $e$ is sufficiently ``close'' to one of the hypersurfaces along
which we have assumed the hyperbolicity condition \eqref{hyperb}.
Provided $|e| >0$ which is the case if $e$ is included in one of the slices of the foliation,
 we associate to $e$ the function $\var_e: \RR \to \RR$, as defined earlier. 
Recall the following hyperbolicity condition which holds along the triangulation since the 
spacelike elements are included in the spacelike slices: 
\be
\label{HYP}
\barc \leq \del_u  \var_{e_K^\pm}(\ubar) \leq \cbar,  \qquad \quad K \in \TT^h.
\ee

We introduce the finite volume method by formally averaging the conservation law \eqref{LR.1}
over each element $K \in \TT^h$ of the triangulation, as follows. Applying Stokes theorem with a smooth solution $u$
to \eqref{LR.1}, we get
$$
0 = \int_K d(\om(u)) = \int_{\dK} i^*\om(u).
$$
Then, decomposing the boundary $\dK$ into its parts $\ekp, \ekm$, and $\dKz$ we find
\be
\label{LM.3}
\int_{\ekp} i^* \om(u) - \int_{\ekm} i^* \om(u) + \sumK \int_{e^0} i^*\om(u) = 0.
\ee
Given the averaged values $\ukm$ along $\ekm$ and $\ukezm$ along $\ekz \in \dKz,$ we need
an approximation $\ukp$ of the average value of the solution $u$ along $\ekp$.
To this end, the second term in \eqref{LM.3} can be approximated by
$$
\int_{\ekm} i^* \om(u) \approx \int_{\ekm} i^* \om(\ukm) = |\ekm|\var_{\ekm}(\ukm)
$$
and the last term by
$$
\int_{\ekz}i^*\om(u) \approx \qq (\ukm,\ukezm),
$$
where the \emph{total discrete flux} $\qq:\RR^{2}\to \RR$ (i.e., a scalar-valued function) must be prescribed.

Finally, the proposed version of the {\sl finite volume method} for the conservation law \eqref{LR.1} takes the form
\be
\label{LM.4}
\int_{\ekp} i^*\om(\ukp) = \int_{\ekm} i^*\om(\ukm) - \sumK  \qq (\ukm,\ukezm)
\ee
or, equivalently,
\be
\label{LM.5}
|\ekp| \vep (\ukp) =  |\ekm| \vem(\ukm) -  \sumK \qq (\ukm,\ukezm).
\ee

We assume that the functions $\qq$ satisfy the following natural assumptions for all $\ubar,\vbar \in \RR:$
\begin{itemize}

\item \emph{Consistency property :}
	\be
	\label{LM.3.1}
	\qq(\ubar,\ubar) = \int_{e^0} i^*\om(\ubar).
	\ee
	
\item \emph{Conservation property :}
	\be
	\label{LM.3.2}
	\qq(\vbar,\ubar) = -q_{K_{e^0},e^0}(\ubar,\vbar).
	\ee
	
\item \emph{Monotonicity property :}
	\be
	\label{LM.3.3}
	\del_{\ubar} \qq(\ubar,\vbar) \geq 0, \qquad \del_{\vbar}\qq(\ubar,\vbar)  \leq 0.
	\ee	
\end{itemize}

We note that, in our notation, there is some ambiguity with the orientation of the faces of the triangulation. 
To complete the definition of the scheme we need to specify the discretization of the initial data
and we define constant initial values $u_{K,0} = \ukm$ (for $K \in \Tcal^h_0$)
associated with the initial slice $H_0$ by setting
\be
\label{datai}
\int_{e_K^-} i^* \omega(u_K^-):= \int_{e_K^-} i^* \omega(u_0), \qquad  \, e_K^- \subset H_0.
\ee
Finally, we define a piecewise constant function $u^h : M \to \RR$ by setting for every element $K \in \TT^h$
\be
\label{uh}
u^h(x) = \ukm, \qquad x \in K.
\ee

It will be convenient to introduce $N_K := \# \del^0 K$,
 the total number of ``vertical'' neighbors of an element $K \in \Tcal^h$, 
which we suppose to be uniformly bounded. 
For definiteness, we fix a finite family of local charts covering the manifold $M$, and
we assume that the parameter $h$ coincides with the largest diameter of faces $e_K^\pm$ of
elements $K \in \Tcal^h$,
where the diameter is computed with the Euclidian metric expressed in the chosen local coordinates
(which are fixed once for all and, of course, overlap in certain regions of the manifold).

For the sake of stability we will need to restrict the time-evolution
and impose the following version of the Courant-Friedrich-Levy condition: for all $K \in \TT^h$,
\be
\label{CFL}
{N_K \over |\ekp|} \max_{e^0 \in \del^0 K} \sup_u \Big| \int_{e^0} \del_u \omega(u) \Big|  < \inf_u \del_u\vep,
\ee
in which the supremum and infimum in $u$ are taken over the range of the initial data.

We then assume the following conditions on the family of triangulations:
\be
\label{AS2}
\lim_{h \to 0} {\tau_{\max}^2 + h^2 \over \tau_{\min}} = 
\lim_{h \to 0}
{\tau_{\max}^2 \over h} = 0 
\ee 
where $\tau_{\max} := \max_i (t_{i+1} - t_i)$ and $\tau_{\min} := \min_i (t_{i+1} - t_i)$. 
For instance, these conditions are satisfied if $\tau_{\max}$, $\tau_{\min}$, and $h$ 
vanish at the same order.

Our main objective in the rest of this paper is to prove the convergence of the above scheme towards
an entropy solution in the sense defined in the previous section.


\subsection{A convex decomposition}

Our analysis of the finite volume method relies on a decomposition of \eqref{LM.5} into essentially one-dimensional schemes.
This technique goes back to Tadmor \cite{Tadmor}, Coquel and LeFloch \cite{CL1}, and 
Cockburn, Coquel, and LeFloch \cite{CCL1}. 

By applying Stokes theorem to \eqref{LR.3} with an arbitrary $\ubar \in \RR$, we have
$$
\aligned
0 & = \int_K d(\om(\ubar))=\int_{\dK} i^*\om(\ubar)
\\
  & = \int_{\ekp} i^* \om(\ubar)- \int_{\ekm} i^* \om(\ubar)
       + \sumK  \qq(\ubar,\ubar).
\endaligned
$$
Choosing $\ubar= u_K^-$, we deduce the identity
\be
\label{ident}
|\ekp| \vep(\ukm) = |\ekm| \vem(\ukm)
- \sumK \qq(u_K^-, u_K^-),
\ee
which can be combined with \eqref{LM.5} so that
$$
\aligned
& \vep( \ukp)
\\
& = \vep( \ukm) -  \sumK {1 \over |\ekp|} \Big( \qq(\ukm,\ukezm) - \qq(\ukm,\ukm) \Big)
\\
& =  \sumK \left(
 {1 \over N_K} \vep( \ukm) - {1 \over |\ekp|} \Big( \qq(\ukm,\ukezm) - \qq(\ukm,\ukm) \Big) \right).
\endaligned
$$

By introducing the intermediate values $\tu$ given by 
\be
\label{LM.7}
\vep(\tu) :=  \vep( \ukm) -  {N_K \over |\ekp|} \Big( \qq(\ukm,\ukezm) - \qq(\ukm,\ukm) \Big),
\ee
we arrive at the desired {\sl convex decomposition}
\be
\label{CD}
\vep(\ukp)=  {1 \over N_K} \sumK \vep(\tu).
\ee

Given any entropy pair $(U,\Om)$ and any hypersurface $e\subset M$ satisfying $|e| >0$ we introduce
the averaged entropy flux along $e$ defined by
$$
\veO(u):= \dashint_e i^* \Om(u).
$$
Obviously, we have $\veo(u) = \var_{e}(u)$.

\begin{lemma}
\label{convx}
For every convex entropy flux $\Omega$ one has
\be
\label{key94}
\vep^\Omega(\ukp) \leq  {1 \over N_K}  \sumK \vep^\Omega(\tu).
\ee
\end{lemma}

The proof below will actually show that the function $ \vep^{\Om}\circ ( \vep^{\om})^{-1}$ is convex.

\begin{proof} It suffices to show the inequality for the entropy flux themselves, and then to average
this inequality over $e$. So, we need to check:
\be
\label{318}
\Omega(\ukp) \leq  {1 \over N_K}  \sumK \Omega(\tu).
\ee
Namely, we have
$$
\aligned
&  {1 \over N_K}  \sumK \big( \Omega(\tu) - \Omega(\ukp) \big)
\\
& =  {1 \over N_K}
\sumK \big( \omega(\ukp) - \omega(\tu) \big) \del_u U(\ukp)
  +
 {1 \over N_K}  \sumK D_{K, e^0},
\endaligned
$$
with
$$
D_{K,e^0} : = \int_0^1 \del_{uu}U(\ukp) \Big(
\omega( \tu + a (\ukp - \tu)) - \omega(\tu) \Big) \, (\ukp - \tu) \, da.
$$
In the right-hand side of the above identity,
the former term vanishes identically in view of \eqref{LM.7}
while
the latter term is non-negative since $U(u)$ is convex in $u$ and $\del_u \omega$ is a positive $n$-form.
\end{proof}


\section{Discrete stability estimates}
\label{discret}

\subsection{Entropy inequalities}

Using the convex decomposition \eqref{CD}, we can derive a discrete version of the entropy inequalities.

\begin{lemma}[Entropy inequalities for the faces]
\label{lem-1}
For every convex entropy pair $(U,\Om)$ and all $K \in\Tcal^h$ and $\ekz\in \del^{0}K$, 
there exists a family of numerical entropy flux functions
$\QQ : \RR^{2}\to\RR$ satisfying the following conditions  for all $u,v \in \RR$:
\begin{itemize}
\item $\QQ$ is consistent with the entropy flux $\Om$:
\be
\label{CONST}
\QQ(u,u ) = \int_{e^0} i^* \Om(u).
\ee
\item Conservation property:
\be \label{CONSV}
\QQ(u,v ) = -{Q}_{K_{\ekz},\ekz}(v,u).
\ee
\item Discrete entropy inequality: with the notation introduced earlier, the finite volume scheme satisfies
       \be \label{DEI}
       \aligned
       \vepO&(\tu) - \vepO(\ukm)
       + {N_K \over |\ekp|}  \Big( \QQ(\ukm,\ukezm) -\QQ(\ukm,\ukm) \Big) \leq 0.
	\endaligned
	\ee
\end{itemize}
\end{lemma}

Combining Lemma ~\ref{convx} with the above lemma immediately implies:

\begin{lemma}[Entropy inequalities for the elements]
For each $K \in \Tcal^h$ one has the inequality
\be
\label{DEI4}
|e_K^+| \, \big( \vepO (\ukp) - \vepO(\ukm) \big)
+ \sumK \big( Q(\ukm,\ukezm) -Q(\ukm,\ukm) \big) \leq 0.
\ee
\end{lemma}

\begin{proof}[Proof of Lemma~\ref{lem-1}]

{\it Step 1.} For  $u,v\in \RR$ and $  e^0\in \dKz$ we  introduce the notation
$$
\Hkez(u,v) :=\vep( u) - {N_K \over |\ekp|}  \Big( \qq(u,v) - \qq(u,u) \Big).
$$
Observe that
$$
 \Hkez(u,u) = \vep( u).
$$
We claim that $\Hkez$ satisfies the following properties:
\be
\frac{\del}{\del u} \Hkez(u,v)\ge 0,
\qquad
\frac{\del}{\del v}\Hkez(u,v) \ge 0.
\ee
The proof of the second property is
immediate by the monotonicity property \eqref{LM.3.3}, whereas, for the first one, we use the
CFL condition \eqref{CFL} together with the monotonicity property \eqref{LM.3.3}. From the definition of $\Hkez(u,v)$,
we observe that
$$
\Hkez(u,u_{K_{\ekz}}) = \big( 1- \sum_{\ekz\in \dKz}  \alpha_{K,\ekz}\big) \vep (u) + \sum_{\ekz\in \dKz}  \alpha_{K,\ekz} \vep (u_{K_{\ekz}}),
$$
where
$$
\alpha_{K,\ekz}:= {1 \over |\ekp|} \frac{\qq(u,u_{K_{\ekz}})- \qq(u,u)}{\vep(u)- \vep(u_{K_{\ekz}})}.
$$
This gives a convex combination of $\vep(u)$ and $ \vep(u_{K_{\ekz}})$. Indeed, 
by the monotonicity property \eqref{LM.3.3}  we have $\sum_{\ekz \in \dKz} \alpha_{K,\ekz} \geq 0$ and 
the CFL condition \eqref{CFL}  gives us
$$
\sum_{\ekz \in \dKz} \alpha_{K,\ekz}
\leq
\sum_{\ekz \in \dKz} {1 \over |\ekp|} \Big| \frac {\qq(u,u_{K_{\ekz}})- \qq(u,u)}{\vep(u) - \vep(u_{K_{\ekz}})}\Big|  \leq 1.
$$

\

{\it Step 2.}
It is sufficient to establish the entropy inequalities for the family of Kruzkov's entropies $\Omegabf$.
In connection with this choice, we introduce the numerical version of Kruzkov's entropy flux
$$
\Qbf(u,v,c):= \qq(u\vee c,v \vee c)- \qq(u\wedge c,v\wedge c),
$$
where $a\vee b= \max (a,b)$ and $a\wedge b= \min (a,b)$. Observe that $\QQ(u,v)$ satisfies the first two properties of the lemma
with the entropy flux replaced by the Kruzkov's family of entropies $\Omega=\Omegabf$ defined in \eqref{KRZ}.

First, we observe

\be \label{LM.H1}
    \begin{aligned}
    &\Hkez(u \vee c,v \vee c) - \Hkez(u \wedge c,v\wedge c)
    \\[5pt]
    &=\vep(u\vee c) - {N_K \over |\ekp|} \big( \qq(u\vee c,v\vee c)
    - \qq(u\vee c,u\vee c)\big) \\[5pt]
    & \quad - \Big( \vep(u\wedge c) -  {N_K \over |\ekp|} \big(\qq(u\wedge c,v\wedge c)
    - \qq (u\wedge c,u\wedge c)\big) \Big) \\[5pt]
      &= \var^{\Omega}_{\ekp}( u ,c) - {N_K \over |\ekp|} \Big( \Qbf(u,v,c) -    \Qbf(u,u,c)\Big),
   \end{aligned}
\ee
where we used
$$
\vep(u\vee c) - \vep(u\wedge c)= \dashint_{\ekp}i^*\Omegabf(u,c) =
\var^{\Omega}_{\ekp}( u ,c).
$$

Second, we check that for $u= \ukm$, $v= \ukezm$ and for any $c\in \RR$
\be
\label{LM.H2}
\Hkez(\ukm \vee c,\ukezm \vee c) - \Hkez(\ukm \wedge c,\ukezm\wedge c)
 \geq \var_{\ekp}^{\Omega}( \tu,c).
 \ee
To prove \eqref{LM.H2} we observe that 
$$
  \begin{array}{l}
     H_{K,e^0}(u,v) \vee H_{K,e^0}(\lambda,\lambda) \leq
     H_{K,e^0}(u\vee\lambda,v\vee\lambda), \\[5pt]
     H_{K,e^0}(u,v) \wedge H_{K,e^0}(\lambda,\lambda) \geq
     H_{K,e^0}(u\wedge\lambda,v\wedge\lambda),
   \end{array}
$$
where  $\Hkez$ is monotone in both variables. Since $\vep$ is monotone, we have
$$
\aligned 
& \Hkez(\ukm \vee  c,\ukezm \vee c) - \Hkez(\ukm \wedge c,\ukezm \wedge c)
\\
    &  \geq  \Big| \Hkez(\ukm,\ukezm)  -  \Hkez(c,c) \Big|
    =  \Big|  \vep(\tu)   - \vep(c) \Big|
    \\
    &  =
    \sgn\big(\vep(\tu)- \vep(c)\big) \big(\vep(\tu)- \vep(c)\big)
 \\
  & = \sgn\big( \tu - c \big)\big(\vep(\tu)- \vep(c)\big)
 =
    \varbO( \tu,c).
    \endaligned 
$$
Combining this identity with \eqref{LM.H1} (with  $u=\ukm$, $v=\ukezm $), we obtain the following
inequality for the Kruzkov's entropies
$$
\aligned
\var_{\ekp}^{\Omega}&(\tu,c) - \var_{\ekp}^{\Omega}(\ukm,c)
     + {N_K \over |\ekp|} \Big( \Qbf(u,v,c) -    \Qbf(u,u,c)\Big) \leq 0.
\endaligned
$$
As already noticed, this inequality implies a similar inequality for all convex entropy flux fields 
and this completes the proof.
\end{proof}


If $V$ is a convex function, then a \emph{modulus of convexity} for $V$ is any positive real
$\beta < \inf V''$, where the infimum is taken over the range of data under consideration.
We have seen in the proof of Lemma \ref{convx} that $\var_e^{\Om}\circ (\var_e^{\om})^{-1}$ is convex for
every spacelike hypersurface $e$ and every convex function $U$ (involved in the definition of $\Omega$).

\begin{lemma}[Entropy balance inequality between two hypersurfaces]
\label{POC.i}
For $K \in\Tcal^h$, let $\beta_{e_K^{+}}$ be a modulus of convexity for
 the function $\vepO \circ \big(\vepo\big)^{-1}$ and set $\beta = \min_{K \in \TT^h} \beta_{\ekp}$. Then, for $i \leq j$ one has
\label{PC-3}
\be
\label{FVM.20}
\aligned
& \sum_{K\in\Tcal_{t_j}^h} |\ekp| \vepO(\ukp)
  + \sum_{\substack{K\in\Tcal^h_{[t_i, t_j)}\\ \ekz\in\dKz}} \frac{\beta}{2 N_K} |\ekp| \big| \tu - \ukp \big|^2
\\
& \leq \sum_{K\in\Tcal_{t_i}^h} |\ekm| \vemO(\ukm),
\endaligned
\ee
where $\Tcal^h_{t_i}$ is the subset of all elements $K$ satisfying $e_K^- \in H_{t_i}$
while $\Tcal^h_{[t_i, t_j)} := \bigcup_{i \leq k < j} \Tcal^h_{t_k}$.
\end{lemma}

We observe that the numerical entropy flux terms no longer appear in \eqref{FVM.20}.

\begin{proof}
Consider the discrete entropy inequality \eqref{DEI}. Multiplying by $|\ekp|/N_K$ and summing in $K\in\Tcal^h$, $\ekz\in\dKz$ gives
\be
\label{FVM.21}
\aligned
&  \sum_{\substack{K\in\Tcal^h\\ \ekz\in\dKz}}
{|\ekp| \over N_K} \vepO (\tu )  -  \sum_{K\in\Tcal^h}|\ekp|\vepO(\ukm)
\\
& \, +  \sum_{\substack{K\in\Tcal^h\\ \ekz\in\dKz}} \big(\QQ(\ukm,\ukezm ) -\QQ(\ukm,\ukm) \big)
\leq 0.
\endaligned
\ee
Next, observe that the conservation property \eqref{CONSV} gives
\be
\label{FVM.22}
  \sum_{\substack{K\in\Tcal^h\\ \ekz\in\dKz}}  \QQ(\ukm,\ukezm ) = 0.
\ee
So \eqref{FVM.21} becomes
\be
\label{FVM.21i}
\aligned
&  \sum_{\substack{K\in\Tcal^h\\ \ekz\in\dKz}} {|\ekp| \over N_K} \vepO (\tu )  -  \sum_{K\in\Tcal^h}|\ekp|\vepO(\ukm)
\\
& \, -  \sum_{\substack{K\in\Tcal^h\\ \ekz\in\dKz}} \QQ(\ukm,\ukm)
\leq 0.
\endaligned
\ee

Now, if $V$ is a convex function, and if $v =  \sum_{j}\alpha_{j} v_{j}$ is a convex combination of $v_{j}$,
then an elementary result on convex functions gives
$$
V(v) + \frac{\beta}{2} \sum_{j}\alpha_{j} |v_{j}- v|^{2} \le \sum_{j} \alpha_{j} V(v_{j}),
$$
where $\beta =\inf V''$, the infimum being taken over all $v_j$. We apply this inequality with  $ v= \vep(\ukp)$
and
$V = \vep ^{\Om}\circ (\vep^\omega)^{-1}$, which is convex.

Thus, in view of the convex combination \eqref{CD} and
by multiplying the above inequality by $|\ekp|$ and then summing in $K\in\Tcal^h$, we obtain
$$
\aligned
&  \sumkez |\ekp| \vepO(\ukp) +
  \sumkez {\beta \over 2}\, {|\ekp| \over N_K} \, |\tu - \ukp|^2
  \\
  &  \leq \sumkez {|\ekp|\over N_K} \vepO(\tu).
    \endaligned
$$
Combining the result with \eqref{FVM.21i}, we find
\be
\label{FVM.23}
\aligned
& \sum_{K\in\Tcal^h} |\ekp| \vepO(\ukp) - \sum_{K \in \Tcal^h} |\ekp| \vepO(\ukm)
 + \sumkez {\beta \over 2} {|\ekp| \over N_K} |\tu - \ukp|^2
\\
&  \leq  \sumkez  \QQ(\ukm,\ukm).
\endaligned
\ee
Using finally the identity
$$
\aligned
 0 &=\int_{K} d(\Om(\ukm))= \int_{\del K} i^* \Om(\ukm)
\\
& = |\ekp| \vepO(\ukm) -  |\ekm| \vemO(\ukm)+ \sum_{\ekz\in\dKz} \QQ(\ukm,\ukm),
\endaligned
$$
we obtain the desired inequality, after further summation over all of 
the elements $K$ within two arbitrary hypersurfaces. 
\end{proof}

We apply Lemma~\ref{POC.i} with a specific choice of entropy function $U$ and obtain the following uniform estimate.

\begin{lemma}[Global entropy dissipation estimate]
The following global estimate of the entropy dissipation holds:
\be
\label{EDE}
\sum_{\substack{K\in\Tcal^h\\ \ekz\in\dKz}} \frac{|\ekp|}{N_K}\big| \tu - \ukp \big| ^2
\lesssim C \,  \int_{H_0} i^* \Omega(u_0)
\ee
for some uniform constant $C>0$, which only depends upon the flux field and the sup-norm of the initial data,
and where $\Omega$ is the $n$-form entropy flux field associated with the quadratic entropy function $U(u) = u^2/2$.
\end{lemma}

\begin{proof} We apply the inequality \eqref{FVM.20} with the choice $U(u)= u^2$
$$
\aligned
0&\geq \sum_{K\in\Tcal^h} ( |\ekp| \vepO(\ukp) -|\ekm| \vemO(\ukm) )
 +  \sum_{\substack{K\in\Tcal^h\\ \ekz\in\dKz}} \
 \frac{\beta}{2} {|\ekp| \over N_K} \big| \tu - \ukp \big|^2.
\endaligned
$$
After summing up in the ``vertical'' direction and keeping only the contribution of the elements
$K \in \Tcal^h_0$ on the initial hypersurface $H_0$,
we find
$$
\aligned
  \sum_{\substack{K\in\Tcal^h\\ \ekz\in\dKz}}
  {|\ekp| \over N_K} \beta \big| \tu - \ukp \big|^2
  \leq
{2 \over \beta} \, \sum_{K\in\Tcal_0^h} |\ekm| \vemO(u_{K,0}).
\endaligned
$$
Finally, we observe that, for some uniform constant $C>0$,
$$
\sum_{K\in\Tcal_0^h} |\ekm| \vemO(u_{K,0}) \leq C \, \int_{H_0} i^* \Omega(u_0).
$$
These expressions are essentially $L^2$ norm of the initial data, and the above inequality 
can be checked by fixing a reference volume form on the initial hypersurface $H_0$ 
and using the discretization \eqref{datai} of the initial data $u_0$. 
\end{proof}


\subsection{Global form of the discrete entropy inequalities}

We now derive a global version of the (local) entropy inequality \eqref{DEI}, 
i.e. we obtain a discrete version of the entropy inequalities arising in the very 
definition of entropy solutions.

One additional notation is necessary to handle ``vertical face'' of the triangulation: we 
fix a reference field of non-degenerate $n$-forms $\omegab$ on $M$ which will be used to measure the ``area'' 
of the faces $e^0 \in \del K^0$. This is necessary in our convergence proof, only, not in the formulation 
of the finite volume method. So, for every $K \in \Tcal^h$ we define 
\be
\label{e-def2}
|e^0|_{\omegab} := \int_{e^0} i^*\omegab \quad \text{ for faces } e^0 \in \del^0 K
\ee
and the non-degeneracy condition means that $|e^0|_{\omegab} >0$.

Given a test-function $\psi$ defined on $M$ and a face $e^0 \in \del^0 K$ of some element,
 we introduce the following averages
$$
\psi_{\ekz} := {\int_{e^0} \psi \, i^* \omegab \over \int_{e^0} i^* \omegab}, 
\qquad
\psi_{\dKz} := {1 \over N_K} \sum_{\ekz\in\dKz}  \psi_{e^0}, 
$$
where, for the first time in our analysis, we use the reference $n$-volume form $\omegab$. 

\begin{lemma}[Global form of the discrete entropy inequalities]
\label{PC-5}
Let $\Om$ be a convex entropy flux field, and let $\psi$ be a non-negative test-function
supported away from the hypersurface $t=T$.
Then, the finite volume approximations satisfy the global entropy inequality
\be
\label{FVM.31}
\aligned
& - \sum_{K \in \TT^h} \int_K d(\psi \Om ) (\ukm)
    - \sum_{K \in\Tcal^h_0} \int_{\ekm} \psi \, i^* \Om(u_{K,0})
  \\
&
\leq A^h(\psi) + B^h(\psi) + C^h(\psi), 
\endaligned 
\ee
with 
$$
\aligned 
& A^h(\psi) := \sumkez {|\ekp| \over N_K} \big( \psi_{\dKz} -  \psi_{\ekz} \big) \, \big( \vepO(\tu) - \vepO(\ukp) \big), 
\\
& B^h(\psi) := \sumkez \int_{\ekz} \big( \psi_{\ekz} - \psi \big) \, i^* \Om(\ukm), 
\\
& C^h(\psi) := -\sum_{K \in\Tcal^h} \int_{\ekp} \big( \psi_{\dKz} - \psi \big) \, \big( i^* \Om(\ukp) - i^*\Om(\ukm) \big).
\endaligned
$$
\end{lemma}

\proof
From the discrete entropy inequalities \eqref{DEI}, we obtain
\be
\label{FVM.36}
\begin{aligned}
& \sumkez {|e_K^+| \over N_K} \psi_{\ekz} \Big( \vepO (\tu ) - \vepO(\ukm)\Big)
\\
& \quad + \sumkez  \psi_{e^0}\ \big( \QQ(\ukm,\ukezm ) - \QQ(\ukm,\ukm) \big)
 \leq 0.
\end{aligned}
\ee
Thanks the conservation property \eqref{CONSV}, we have
$$
\sumkez \psi_{\ekz} \QQ(\ukm,\ukezm ) = 0
$$
and, from the consistency property \eqref{CONST}, 
$$
\aligned
&  \sumkez  \psi_{\ekz}\QQ(\ukm,\ukm ) =  \sumkez \psi_{\ekz} \int_{\ekz}i^*\Om(\ukm)
\\
& \quad =  \sumkez \int_{\ekz} \psi \, i^* \Om(\ukm)
    +  \sumkez \int_{\ekz} (\psi_{\ekz} - \psi)  i^* \Om(\ukm).
\endaligned
$$
Next, we observe 
$$
\aligned
& \sumkez {|\ekp| \over N_K} \psi_{e^0} \, \vepO(\tu)
\\
& = \sumkez {|\ekp| \over N_K} \psi_{\dKz} \vepO(\tu)
  +\sumkez {|\ekp| \over N_K} (\psi_{\ekz} - \psi_{\dKz} )  \vepO (\tu)
\\
&\geq \sum_{K \in\Tcal^h} |\ekp| \psi_{\dKz} \vepO(\ukp)
      +\sumkez  {|\ekp| \over N_K} (\psi_{\ekz} - \psi_{\dKz} ) \vepO ( \tu ), 
\endaligned
$$ 
where, we used the inequality \eqref{key94} and the convex combination \eqref{CD}.
In view of 
$$
\sumkez {|\ekp| \over N_K} \psi_{\ekz} \, \vepO(  \ukm)
=   \sum_{K\in\Tcal^h}|\ekp| \psi_{\dKz} \, \vepO(  \ukm), 
$$
the inequality \eqref{FVM.36} becomes
\be
\label{FVM.39}
\aligned
& \sum_{K \in\Tcal^h} |\ekp|\psi_{\dKz}  \Big(
 \vepO(\ukp) - \vepO(\ukm)\Big)
 -\sumkez \int_{\ekz} \psi \,  i^*\Om (\ukm)
\\
& \leq   -\sumkez {|\ekp| \over N_K} (\psi_{\ekz} - \psi_{\dKz}) \vepO ( \tu )
 + \sumkez \int_{\ekz} (\psi_{\ekz} - \psi)\, i^* \Om(\ukm).
\endaligned
\ee

Note that the first term in \eqref{FVM.39} can be written as
$$
\aligned
&\sum_{K \in\Tcal^h}|\ekp| \psi_{\dKz}  \Big(
\vepO(\ukp) -\vepO (\ukm)\Big)
 \\
 &= \sum_{K \in\Tcal^h} \int_{\ekp} \psi (i^* \Om(\ukp) - i^*\Om(\ukm))
  + \sum_{K \in\Tcal^h} \int_{\ekp} (\psi_{\dKz} - \psi) \,   (i^* \Om(\ukp) - i^*\Om(\ukm)).
\endaligned
$$
We can sum up (with respect to $K$) the identities 
\be
\nonumber
\aligned
  &\int_{K} d(\psi \Om )(\ukm) = \int_{\del K} \psi \, i^*\Om(\ukm)
\\
&   =\int_{\ekp} \psi \, i^*\Om(\ukm) - \int_{\ekm} \psi \,i^*\Om(\ukm)
 + \sum_{\ekz\in\dKz} \int_{\ekz} \psi \, i^* \Om(\ukm)
\endaligned
\ee
and combine them with the inequality \eqref{FVM.39}. Finally, we arrive at the desired conclusion 
by noting that
$$
\aligned
& \sum_{K \in\Tcal^h} \Big(  \int_{\ekp} \psi \, i^* \Om(\ukp) -   \int_{\ekm} \psi \, i^* \Om(\ukm)\Big)
  = - \sum_{K \in\Tcal^h_0} \int_{\ekm} \psi \, i^* \Om(u_{K,0}).
\endaligned
$$
\endproof


\section{Convergence and well-posedness results}
\label{wellp}

We are now in a position to establish: 

\begin{theorem}[Convergence of the finite volume method]
\label{convergencefvm}
Under the assumptions made in Section~\ref{finit} and provided the flux field is geometry-compatible,
the family of approximate solutions $u^h$ generated by the finite volume scheme
converges (as $h \to 0$) to an entropy solution
of the initial value problem \eqref{LR.1}, \eqref{cond1}.
\end{theorem}

Our proof of convergence of the finite volume method 
can be viewed as a generalization to spacetimes of the technique introduced by
Cockburn, Coquel and LeFloch \cite{CCL00,CCL2} for the (flat) Euclidean setting and already 
extended to Riemannian manifolds by Amorim, Ben-Artzi, and LeFloch \cite{ABL}
and to Lorentzian manifolds by Amorim, LeFloch, and Okutmustur \cite{ALO}.  

We also deduce that:

\begin{corollary}[Well-posedness theory on a spacetime]
\label{main}
Let $\MM=[0,T] \times N$ be a $(n+1)$-dimensional spacetime foliated by $n$-dimen\-sional
hypersurfaces $H_t$ ($t \in [0,T]$) with compact topology $N$ (cf.~\eqref{LR.1}).
Let $\om$ be a geometry-compatible flux field on $\MM$ satisfying the global hyperbolicity condition \eqref{hyperb}. 
An initial data $u_0$ being prescribed on $H_0$,
the initial value problem \eqref{LR.1}, \eqref{cond1} admits
an entropy solution $u \in L^\infty(M)$ which, moreover, has well-defined $L^1$ traces on any
spacelike hypersurface of $M$.
These solutions determines a (Lipschitz continuous) contracting semi-group in the sense
that the inequality
\be
\label{MTHM.1}
\int_{H'} i_{H'}^* \Omegabf\big( u_{H'}, v_{H'} \big)
\leq \int_{H} i^*_H \Omegabf\big( u_H,v_H\big)
\ee
holds for any two hypersurfaces $H,H'$ such that $H'$ lies in the future of $H$, and 
the initial condition is assumed in the weak sense
\be
\label{MTHM.23}
\lim_{t \to 0 \atop t>0} \int_{H_t} \iht \Omegabf\big( u(t), v(t) \big) = \int_{H_0} \ihz \Omegabf(u_0,v_0).
\ee
\end{corollary}

We can also extend a result originally established by DiPerna \cite{DiPerna}
(for conservation laws posed on the Euclidian space)
within the broad class of entropy measure-valued solutions.

\begin{theorem}
\label{MTHM.2}
Let $\om$ be a geometry-compatible flux field on a spacetime $\MM$ satisfying the global hyperbolicity condition \eqref{hyperb}.
Then, any entropy measure-valued solution $\nu$ (see Definition \ref{LR.4}) 
to the initial value problem \eqref{LR.1}, \eqref{cond1} reduces to a Dirac mass at each point, more precisely
\be
\nu= \delta_u,
\ee
where $u \in L^\infty(\MM)$ is the unique entropy solution to the same problem.
\end{theorem}

We omit the details of the proof, since it is a variant of the Riemannian proof given in \cite{BL}.


It remains to provide a proof of Theorem~\ref{convergencefvm}.
Recall that a Young measure $\nu$ allows us to determine all weak-$*$ limits of composite functions $a(u^h)$
for all continuous functions $a$, as $h \to 0$, 
\be
\label{COP.1}
a(u^h) \mathrel{\mathop{\rightharpoonup}\limits^{*}} \langle \nu, a \rangle
=
 \int_\RR a(\lambda) \, d\nu(\lambda). 
\ee

\begin{lemma}[Entropy inequalities for the Young measure]
\label{COP-1}
Let $\nu$ be a Young measure associated with the finite volume approximations $u^h$.
Then, for every convex entropy flux field $\Om$ and every non-negative test-function $\psi$ 
supported away from the hypersurface $t=T$, one has
\be
\label{COP.2}
\int_{M} \langle  \nu, d\psi \wedge \Om(\cdot)\rangle
- \int_{H_{0}} i^*\Om(u_0) \leq 0.
\ee
\end{lemma}

Based on this lemma, we are now in position to complete the proof of Theorem \ref{convergencefvm}.
Thanks to \eqref{COP.2}, we have for all convex entropy pairs
$(U,\Om)$,
$$
d  \langle \nu, \Om(\cdot) \rangle  \leq 0
$$
in the sense of distributions on $M$.
On the initial hypersurface $H_0$ the (trace of the) Young measure $\nu$ coincides with
the Dirac mass $\delta_{u_0}$.
By Theorem~\ref{MTHM.2} there exists a unique function $u \in L^\infty(M)$ (the entropy solution to the
initial-value problem under consideration)
such that the measure $\nu$ coincides with the Dirac mass $\delta_u$.
Moreover, this property also implies that the approximations $u^h$ converge strongly to $u$,
and this concludes the proof of the convergence of the finite volume scheme.

\begin{proof}[Proof of Lemma~\ref{COP-1}]
The proof is a direct passage to the limit in the inequality \eqref{FVM.31},
by using the property \eqref{COP.1} of the Young measure.
First of all, we observe that,  the left-hand side of the inequality \eqref{FVM.31} converges  to
the left-hand side of \eqref{COP.2}.
Indeed, since $\omega$ is geometry-compatible, the first term of interest
$$
\sum_{K \in \TT^h} \int_K d(\psi \Om ) (\ukm)
= 
\sum_{K \in \TT^h} \int_K d\psi \wedge \Om (\ukm)
= 
\int_M d\psi \wedge \Om(u^h) 
$$
converges to $\int_{M} \langle  \nu, d\psi \wedge \Om(\cdot)\rangle$. On the other hand,  
the initial contribution
$$
\sum_{K \in\Tcal^h_0} \int_{\ekm} \psi \, i^* \Om(u_{K,0})
=
\int_{H_0} \psi \, i^*\Om(u_0^h) 
\to \int_{H_{0}} \psi \, i^*\Om(u_0),  
$$ 
in which $u_0^h$ is the initial discretization of the data $u_0$ and converges strongly to $u_0$
since the maximal diameter $h$ of the element  tends to zero.

It remains to check that the terms on the right-hand side of \eqref{FVM.31} vanish in the limit $h \to 0$.
We begin with the first term $A^h(\psi)$.
Taking the modulus of this expression, applying Cauchy-Schwarz inequality, and finally 
using the global entropy dissipation estimate \eqref{EDE}, we obtain 
$$
\aligned
|A^h(\psi)|  
&
\leq
\sumkez {|\ekp| \over N_K} |\psi_{\dKz} -\psi| |\tu - \ukm|
\\
&
\leq
\Big(\sumkez {|\ekp|\over N_K} |\psi_{\dKz} -\psi|^2\Big)^{1/2} \Big( \sumkez {|\ekp| \over N_K} |\tu - \ukm| ^2 \Big)^{1/2}
\\
& 
\leq \Big( \sumkez {|\ekp|\over N_K}  (C \, (\tau_{\max} + h))^2\Big)^{1/2} \Big( \int_{H_0}i^*\Om(u_0)\Big) ^{1/2}, 
\endaligned
$$
hence
$$
\aligned
|A^h(\psi) | 
& \leq C' \, (\tau_{\max} + h) \Big( \sumk |\ekp| \Big)^{1/2}
\leq C'' \,  {\tau_{\max} + h  \over (\tau_{\min})^{1/2}}. 
\endaligned
$$
Here, $\Om$ is associated with the quadratic entropy and 
have used the fact that $|\psi_{\dKz} - \psi| \leq C \, (\tau_{\max} + h)$. 
Our conditions \eqref{AS2} imply the upper bound for $A^h(\psi)$ tends to zero with $h$.

Next, we rely on the regularity of $\psi$ and $\Om$ and estimate the second term on the right-hand side of \eqref{FVM.31}. 
By setting  
$$
C_{\ekz} := {\int_{\ekz} i^* \Omega(u_K^-) \over  \int_{\ekz} i^* \omegab}, 
$$
we obtain 
$$
\aligned
|B^h(\psi)| 
& = \Big| \sumkez \int_{\ekz} (\psi_{\ekz} - \psi)\, \Big( i^* \Om(\ukm) - C_{\ekz} i^* \omegab \Big) \Big| 
\\
& \leq \sumkez  \sup_K |\psi_{\ekz} - \psi| \, \int_{\ekz} \Big| i^* \Om(\ukm) - C_{\ekz} i^* \omegab \Big|
\\
&
\leq C \, \sumkez  (\tau_{\max} + h)^2 \, |e^0|_{\omegab},  
\endaligned
$$
hence
$$
|B^h(\psi)| 
\leq C \, {(\tau_{\max} + h)^2  \over h}. 
$$
Again, our assumptions imply the upper bound for $B^h(\psi)$ tends to zero with $h$. 

Finally, consider the last term in the right-hand side of \eqref{FVM.31}
$$
\aligned
|C^h(\psi)| 
& \leq 
\sum_{K \in\Tcal^h} |\ekp| \sup_K |\psi_{\dKz} -\psi| \int_{\ekp}  |i^* \Om(\ukp) - i^*\Om(\ukm)|,  
\endaligned
$$
using the modulus defined in the beginning of Section~2. 
In view of the inequality \eqref{318}, we obtain
$$
\aligned
|C^h(\psi)| 
& \leq
C \,  \sumkez {|\ekp| \over N_K} |\psi_{\dKz} -\psi| \,  \big| \tu - \ukm \big|, 
\endaligned
$$
and it is now clear that $C^h(\psi)$ satisfies the same estimate as the one we derived for $A^h(\psi)$.  
\end{proof} 


\section*{Acknowledgments}

This paper was completed when the first author visited the Mittag-Leffler Institute
in the Fall 2008 during the Semester Program ``Geometry, Analysis, and General Relativity''
organized by L. Andersson, P. Chrusciel, H. Ringstr\"om, and R. Schoen.  
This research was partially supported by the A.N.R. (Agence Nationale de la Recherche)
through the grant 06-2-134423 entitled {\sl ``Mathematical Methods in General Relativity''} (MATH-GR),
and by the Centre National de la Recherche Scientifique (CNRS).



\begin{thebibliography}{10}

\bibitem{ABL} \auth{P. Amorim, M. Ben-Artzi, and P.G. LeFloch},
Hyperbolic conservation laws on manifolds. Total variation estimates and the finite volume method,
{\em Meth. Appli. Anal.} 12 (2005), 291--324.

\bibitem{ALO} \auth{P. Amorim, P.G. LeFloch, and B. Okutmustur},
Finite volume schemes on Lorentzian manifolds,
{\em arXiv:0712.122.}

\bibitem{BLN} \auth{C.W. Bardos, A.-Y. Leroux, and J.-C. Nedelec,}
First order quasilinear equations with boundary conditions,
{\em Comm. Part. Diff. Eqns.} 4 (1979), 75Ð78.

\bibitem{BL} \auth{M. Ben-Artzi and P.G. LeFloch,}
The well-posedness theory for geometry compatible hyperbolic conservation laws on manifolds,
{\em Ann. Inst. H. Poincar\'e, Nonlinear Anal.} 24 (2007), 989--1008.

\bibitem{BFL1}  \auth{M. Ben-Artzi, J. Falcovitz, and P.G. LeFloch,}
Hyperbolic conservation laws on the sphere. A geometry compatible finite volume scheme, 
{\em arXiv:0808.2062.} 

\bibitem{CCL00} \auth{B. Cockburn, F. Coquel, and P.G. LeFloch,} 
An error estimate for high-order accurate finite volume methods for scalar conservation laws,
{\em Preprint 91-20, AHCRC Institute,} Minneapolis (USA), 1991.

\bibitem{CCL2} \auth{B. Cockburn, F. Coquel, and P.G. LeFloch,} 
Error estimates for finite volume methods for multidimensional conservation laws, 
{\em Math. of Comput.} 63 (1994), 77--103. 

\bibitem{CCL1} \auth{B. Cockburn, F. Coquel, and P.G. LeFloch,}
Convergence of Þnite volume methods for multi-dimensional conservation laws, 
{\em SIAM J. Numer. Anal.} 32 (1995), 687Ð-705.

\bibitem{CL1} \auth{F. Coquel and P.G. LeFloch,} 
Convergence of finite difference schemes for conservation laws in several space dimensions,
{\em C.R. Acad. Sci. Paris Ser.} I 310 (1990), 455--460.

\bibitem{CL2} \auth{F. Coquel and P.G. LeFloch,} 
Convergence of finite difference schemes for conservation laws in several space dimensions: a general theory,
{\em SIAM J. Numer. Anal.} 30 (1993), 675--700. 

\bibitem{CL3}\auth{F. Coquel and P.G. LeFloch,} 
Convergence of finite difference schemes for conservation laws in several space dimensions: 
the corrected antidiffusive flux approach,
{\em Math. of Comp.} 57 (1991), 169--210. 

\bibitem{DiPerna} \auth{R.J. DiPerna,}
Measure-valued solutions to conservation laws,
{\em Arch. Rational Mech. Anal.} 88 (1985), 223--270.

\bibitem{DL} \auth{F. Dubois and P.G. LeFloch,}
Boundary conditions for nonlinear hyperbolic systems of conservation laws,
{\em J. Differential Equations} 31 (1988), 93Ð122.

\bibitem{EGH} \auth{R. Eymard, T. Gallou\"et, and R. Herbin,}
The finite volume method,
in {\em Handbook of Numerical Analysis~VII,} North-Holland, Amsterdam, 2000, pp.~713--1020.

\bibitem{KL}  \auth{C. Kondo and P.G. LeFloch,}
Measure-valued solutions and well-posedness of multi-dimensional conservation laws in a bounded domain,
{\em Portugal. Math.} 58 (2001), 171--194.

\bibitem{Kroener} \auth{D. Kr\"oner,}
Finite volume schemes in multidimensions,
in {\em Numerical analysis} 1997 (Dundee),  Pitman Res. Notes Math. Ser. 380, Longman, 
Harlow, 1998, pp.~179--192. 

\bibitem{KNR} \auth{D. Kr\"oner, S. Noelle, and M. Rokyta,}
Convergence of higher-order upwind finite volume 
schemes on unstructured grids for scalar conservation laws with several space dimensions,
{\em Numer. Math.} 71 (1995), 527Ð560. 

\bibitem{Kruzkov} \auth{S. Kruzkov,}
First-order quasilinear equations with several space variables,
{\em Math. USSR Sb.} 10 (1970), 217--243.

\bibitem{Kuznetsov} \auth{N.N Kuznetsov,}
Accuracy of some approximate methods for computing the weak solutions of a first-order quasi linear equations,
{\em USSR Comput. Math. Math. Phys.} 16 (1976), 105--119.

\bibitem{Lax} \auth{P.D. Lax,}
{\sl Hyperbolic systems of conservation laws and the mathematical theory of shock waves,}
Regional Conf. Series in Appl. Math., Vol.~11, SIAM, Philadelphia, 1973.

\bibitem{PLF}  \auth{P.G. LeFloch,}
Hyperbolic conservation laws and spacetimes with limited regularity,
{\em Proc. 11th Inter. Confer.} on ``Hyper. Problems: theory, numerics, and applications'',
ENS Lyon, July 17--21, 2006, S. Benzoni and D. Serre ed., Springer Verlag, pp.~679--686.
(See also arXiv:0711.0403.)

\bibitem{LNO}  \auth{P.G. LeFloch, W. Neves, and B. Okutmustur,}
Hyperbolic conservation laws on manifolds. Error estimate for finite volume schemes,
{\em Acta Math. Sinica} (2009). (See also arXiv:0807.4640.)

\bibitem{LO}  \auth{P.G. LeFloch and B. Okutmustur,}
Hyperbolic conservation laws on manifolds with limited regularity,
{\em C.R. Math. Acad. Sc. Paris} 346 (2008), 539--543.

\bibitem{Otto} \auth{F. Otto,}
Initial-boundary value problem for a scalar conservation law, 
{\em C. R. Acad. Sci. Paris S«er. I Math.} 322 (1996), 729Ð-734.

\bibitem{Panov} \auth{E.Y. Panov,}
The Cauchy problem for a Þrst-order quasilinear equation on a manifold, 
{\em Differential Equation} 33 (1997), 257Ð-266.
 
\bibitem{Szepessy} \auth{S. Szepessy,}
Convergence of a shock-capturing streamline diffusion finite element method for a scalar conservation law 
in two space dimensions,
{\em Math. Comp.} 53 (1989), 527--545.  
 
\bibitem{Sz} \auth{A. Szepessy,}
Measure-valued solutions of scalar conservation laws with boundary conditions, 
{\em Arch. Rational Mech. Anal.} 107 (1989), 181Ð-193.
 
\bibitem{Szepessy2} \auth{A. Szepessy,}  
Convergence of a streamline diffusion finite element method for scalar conservation laws with boundary conditions,
{\em RAIRO Model. Math. Anal. Numer.} 25 (1991), 749--782.  

\bibitem{Tadmor} \auth{E. Tadmor,}
Approximate solutions of nonlinear conservation laws,
in {\em Advanced numerical approximation of 
nonlinear hyperbolic equations} (Cetraro, 1997),
Lecture Notes in Math., 1697, Springer, Berlin, 1998,  pp.~1--149.  

\bibitem{TRB} \auth{E. Tadmor, M. Rascle, and P. Bagnerini,}
Compensated compactness for 2D conservation laws,
{\em J. Hyperbolic Differ. Equ.} 2 (2005), 697--712.

\bibitem{WN} \auth{M. Westdickenberg and S. Noelle,}
A new convergence proof for finite volume schemes using the kinetic formulation of conservation laws,
{\em SIAM J. Numer. Anal.} 37 (2000), 742--757.

\end{thebibliography}
\end{document}